\documentclass[reqno,11pt]{amsart}
\usepackage{amssymb}
\usepackage{amsmath}
\usepackage{stmaryrd}
\usepackage{graphicx}
\usepackage{color}
\usepackage{mathrsfs}
\usepackage{cite}
\usepackage[font=small]{caption}
\usepackage{bm}

\usepackage{tikz}
\usepackage{pgfplots}
\pgfplotsset{compat=1.10}
\usepgfplotslibrary{fillbetween}
\usetikzlibrary{patterns}
\newtheorem{theorem}{Theorem}[section]
\newtheorem*{theorem*}{Main result}
\newtheorem{corollary}[theorem]{Corollary}
\newtheorem{lemma}[theorem]{Lemma}
\newtheorem{proposition}[theorem]{Proposition}

\newtheorem{remark}[theorem]{Remark}

\newcommand{\R}{\mathbb{R}}
\newcommand{\N}{\mathbb{N}}
\newcommand{\Z}{\mathbb{Z}}
\newcommand{\I}{\mathrm{i}}
\newcommand{\rd}{\mathrm{d}}

\definecolor{cadmiumgreen}{rgb}{0.0, 0.42, 0.24}
\numberwithin{equation}{section}
\numberwithin{figure}{section}
\usepackage[colorlinks=true]{hyperref}
\hypersetup{urlcolor=blue, citecolor=red, linkcolor=blue}

\begin{document}

\title[Solitary waves to the Whitham equation]{A direct construction of a full family of Whitham solitary waves}

\author{Mats Ehrnstr\"om}
\address{Department of Mathematical Sciences \\ Norwegian University of Science and Technology \\ NO--7491 Trondheim, Norway}
\email{mats.ehrnstrom@ntnu.no}
\author{Katerina Nik}
\address{Faculty of Mathematics\\ University of Vienna \\ Oskar-Morgenstern-Platz 1 \\ A--1090 Vienna\\ Austria}
\email{katerina.nik@univie.ac.at}
\author{Christoph Walker}
\address{Leibniz Universit\"at Hannover\\ Institut f\" ur Angewandte Mathematik \\ Welfengarten 1 \\ D--30167 Hannover\\ Germany}
\email{walker@ifam.uni-hannover.de}
%
\thanks{ME acknowledge the support by grant no. 250070 from the Research Council of Norway. KN is partially supported by the Austrian Science Fund (FWF) project  F\,65.}
\date{\today}
%
%
%
\begin{abstract} Starting with the periodic waves earlier constructed for the gravity Whitham equation, we parameterise the solution curves through relative wave height, and use a limiting argument to obtain a full family of solitary waves. The resulting branch starts from the zero solution, traverses unique points in the wave speed–wave height space, and reaches a singular highest wave at \(\varphi(0) = \frac{\mu}{2}\). The construction is based on uniform estimates improved from earlier work on periodic waves for the same equation, together with limiting arguments and a Galilean transform to exclude vanishing waves and waves levelling off at negative surface depth. In fact, the periodic waves can be proved to converge locally uniformly to a wave with negative tails, which is then transformed to the desired branch of solutions. The paper also contains some proof concerning uniqueness and continuity for signed solutions (improved touching lemma).
\end{abstract}
%
\maketitle
%

\section{Introduction} \label{sec:intro}
\noindent In this remark we construct solitary waves to the nonlinear and nonlocal Whitham equation
\begin{equation*}
	u_t + (Lu+u^2)_x=0. 
\end{equation*}
The real-valued function $u=u(t,x)$ describes the deflection of a water fluid surface from the rest 
position at time $t\geq 0$ and position $x \in \mathbb{R}$, and $L\colon f \mapsto K \ast f$ denotes convolution with the kernel $K$ given by 
\begin{equation*}
	K(x) = \frac{1}{2 \pi} \int_{\mathbb{R}} m(\xi) \exp(ix\xi) \, \rd \xi, \quad \text{ with } \; 
	m(\xi) = \sqrt{\frac{\text{tanh}(\xi)}{\xi}}. 
\end{equation*}
Starting with \cite{EK08} this equation has been extensively studied recently, not the least because of its solutions' interesting qualitative properties, and the features of wave breaking, solitary waves, and highest waves in a single scalar model equation (see below). The Whitham equation is one of the simplest examples of an inherently nonlocal and nonlinear dispersive equation \cite{MR3188389}, which gains its modelling strength from the linear dispersion relation of the finite-depth gravity Euler equations \cite{MR3060183}, where \(m(\xi)\) describes the speed of a  right-propagating steady wave train of frequency \(\xi\). 

The Whitham equation has features of small waves (KdV-type bifurcation of periodic \cite{EK08} and solitary waves \cite{EGW11,MR4072387}, modulational instabilities \cite{MR3298879}, prolonged existence time \cite{EW20whitham}, improved modelling \cite{MR3763731, MR4321411}) and of 'large' waves (symmetry and nodal properties \cite{BP21}, wave breaking \cite{MR3682673}, and highest waves \cite{EW19,TWW20}). A number of additional qualities are known, including exponential decay \cite{BEP16} and experimental modelling properties~\cite{MR3430140}. The difficulties in this equation arise because of the mixing of local and nonlocal inhomogeneous terms, which make precise estimates very demanding, and the coupling between parameters such as wave speed and wave height to the solution itself remains implicit. 

For travelling waves $u(t,x)=\varphi(x-\mu t)$, with $\mu>0$ the sought wave speed, more is known than for the initial-value problem. By integrating once and making use of a Galilean invariance
(see \eqref{eq:galilean} below), the Whitham equation reduces to the integral equation 
\begin{equation}
\label{steadyW}
- \mu \varphi + L \varphi + \varphi^2=0,
\end{equation}
where one seeks functions $\varphi\colon \mathbb{R} \rightarrow \mathbb{R}$ satisfying \eqref{steadyW} pointwise in $\mathbb{R}$. We shall refer to these as solutions of \eqref{steadyW} with 
wave speed $\mu$. In this case, it is known that each period allows a global curve of smooth solutions of bell-shaped form, bifurcating off from the line of zero solutions and stretching continuously (or better) to a so-called highest wave of wave height \(\varphi(0) = \frac{\mu}{2}\) above an undisturbed surface, and of optimal regularity \(C^{1/2}(\R)\) \cite{EW19}. For the solitary case, when \(\varphi(x) \to 0\) as \(x \to \pm \infty\), the first result was in \cite{EGW11} and features a quite advanced construction of small solitary solutions for that and many other equations of negative-order dispersion with the aid of minimisation. This has been generalised in \cite{MR4072387}. Later, proofs based on the implicit function theorem \cite{MR4061635} and centre manifolds \cite{TWW20} have been suggested, and recently Truong, Wahlén and Wheeler used the latter to prove and extend the branch of small solutions through global bifurcation in weighted Sobolev spaces, all the way up to the highest wave. The estimates for the highest waves are the same in the periodic and solitary case, see \cite{EMV21,EW19}. We would like to mention also \cite{JTW21} and \cite{MR4057934}, which although they are for the positively dispersive capillary-gravity Whitham equation, make use of similar techniques as the aforementioned papers. Global estimates of the kind used in our investigation are so far not known in the case of surface tension (in this or other equations with non-explicit solutions). 

In the current paper we give an alternative, less sophisticated but straightforward approach to the same problem, by making use of the theory for periodic waves. While it is a common procedure in dynamical systems to obtain homo- and heteroclinic orbits as limits of periodic orbits, this procedure in water waves is most commonly associated with small-amplitude solutions (see, e.g., \cite{MR766131,Buffoni04a}), whereas we carry it out for the full range of solutions. Our method is constructive, picking periodic waves which can be found analytically as well as numerically, and then obtaining point-wise convergence as the period tends to infinity. While the limiting process eliminates some of the regularity properties of the bifurcation curves, this is compensated for by the method's simplicity and that we may choose the relative heights \(\lambda = \max \varphi/(\mu/2)\) of the initial waves, yielding a family of unique solitary solutions \((\varphi,\mu)\) reaching the highest wave at \(\max \varphi = \frac{\mu}{2}\). Although periodic limiting sequences appear in both \cite{EGW11} and \cite{MR4072387}, the method in the current paper is arguably less pricey and yields a direct convergence with known crest and tails in terms of the wave speed. As has to be, see  \cite{BEP16}, our solutions are symmetric and monotonically exponentially decaying. In Sections~\ref{sec:preliminaries} and~\ref{sec:properties} we give the necessary background and properties of periodic and solitary solutions, whereas in Section~\ref{sec:construction} the actual construction takes place. The main result can be summarised as follows (see Theorem~\ref{thm:main}, Corollary~\ref{cor:galilean}, Propositions~\ref{prop:injective} and \ref{prop:upperbound}).

\begin{theorem*}
Pick \(\lambda \in (0,1]\). The steady Whitham equation~\eqref{steadyW} allows for a sequence of \(P\)-periodic solutions \((\varphi_P,\mu_P)\) of height \(\varphi_P(0) = \lambda \frac{\mu}{2}\) converging pointwise and locally uniformly as \(P \to \infty\)  to a symmetric solution of the same relative wave height \(\varphi_P(0) = \lambda \frac{\mu}{2}\), that decays monotonically on either side of the origin. After a Galilean transformation, this family describes an injective branch \(\lambda \mapsto (\varphi_\lambda, \mu_\lambda) \in C(\R) \times \R\) of solitary solutions to \eqref{steadyW}, which have supercritical wave speed \(\mu \in (1,2)\), are exponentially decaying and reach the highest wave of optimal \(C^{1/2}\)-regularity at \(\lambda = 1\). The speed \(\mu\) may be bounded in terms of the relative height of solutions, converges to \(1\) from above as \(\lambda \to 0\), and all solutions are smooth except for the highest at \(\lambda =1\). 
\end{theorem*}

\noindent Further properties of the solutions, including a slightly improved touching lemma in Proposition~\ref{prop:touching}, are presented in Section~\ref{sec:construction}.


\section{Preliminaries} \label{sec:preliminaries}
\noindent We give in this short section the basic background on the convolution kernel \(K\) and its properties. These properties were established in \cite{EW19}. The kernel $K$ is given by 
\begin{equation*}
	K(x) = (\mathcal{F}^{-1}m)(x) = \frac{1}{2 \pi} \int_{\mathbb{R}} m(\xi) \exp(ix\xi) \, \rd \xi
\end{equation*}
with Whitham symbol 
\begin{equation*}
	m(\xi) = \sqrt{\frac{\text{tanh}(\xi)}{\xi}}= 1-\frac{1}{6}\xi^2+ \mathcal{O}(\xi^4), \qquad \text{ as }  |\xi| \to 0.
\end{equation*}
Let \(\N = \{0,1,2, \ldots\}\).

\begin{lemma}	\label{propertiesWK}
The kernel $K$ is even and completely monotone on $(0,\infty)$, that is,
\[
 {\textstyle \left(-\frac{d}{dx}\right)^j} K(x) > 0, \qquad x > 0, \quad j \in \mathbb{N}.
\]
In particular, $K$ is positive, and strictly decreasing and convex for $x > 0$. It integrates to unity,
\begin{equation*}
	\Vert K \Vert_{L^1(\mathbb{R})} =1,
\end{equation*}
and furthermore satisfies the asymptotic estimates that:\\

\begin{itemize}
\item[(i)] For any given $s_0 \in (0,\pi/2)$ and $n \in \mathbb{N}$,
\[
\vert D_x^n K(x) \vert \lesssim  \exp(-s_0 \vert x\vert), \qquad \vert x \vert \geq \tfrac{1}{2}.\\[10pt]
\]

\item[(ii)] With $K_{\textnormal{reg}}$ infinitely differentiable on $\mathbb{R}$,
\[
K(x) = \frac{1}{\sqrt{2 \pi \vert x\vert}} + K_{\textnormal{reg}}(x), \qquad x \neq 0.\\[10pt]
\]
\end{itemize}
\end{lemma}

\noindent For the proof of Lemma~\ref{propertiesWK}, see \cite[Section 2]{EW19}. We introduce now the \emph{periodised kernel}
\begin{equation}
\label{periodisedWK}
	K_P(x)= \sum_{n\in \mathbb{Z}} K(x+nP)
\end{equation}
for $P\in (0, \infty)$, with the convention $K_{\infty}= K$ if $P=\infty$. By property (i) in Lemma \ref{propertiesWK} this sum is absolutely convergent. $K_P$ can equivalently be written as a Fourier series, 
\begin{equation*}
	K_P(x) = \frac{1}{P} \sum_{n\in \mathbb{Z}} m\Big(  \frac{2 \pi n}{P}\Big) \exp\Big(  \frac{2 \pi \I nx}{P}\Big) .
\end{equation*}

\begin{lemma}\label{propertiesPWK}
The kernel $K_P$ is completely monotone on $(0,P/2)$. Moreover,
\begin{itemize}
\item[(i)]
	 $K_P$ is $P$-periodic, even, positive, strictly decreasing on $(0,P/2)$ and convex on  $(0,P)$. \\[-0.2cm]
\item[(ii)]	
$K_P \in C^{\infty}(\mathbb{R}\backslash P\mathbb{Z} )$. \\[-0.2cm]
\item[(iii)]	
$K_P$ satisfies 
	\begin{equation*}
		K_P(x) = \frac{1}{\sqrt{2 \pi \vert x\vert}} + K_{P,\textnormal{reg}}(x), 
	\end{equation*}
	where $K_{P,\textnormal{reg}}$ is infinitely differentiable on $(-P,P)$. 	
\end{itemize}
\end{lemma}

\noindent For a proof, see \cite[Section 3]{EW19}. As in the introduction, we have the linear operator 
\begin{equation*}
	L\colon f \mapsto K \ast f,
\end{equation*}
which we define on the space of tempered distributions $\mathcal{S}'(\mathbb{R})$ via duality from Schwartz space \(\mathcal{S}(\R)\). One sees from \eqref{periodisedWK} that for a continuous periodic function $f$, the operator $L$ 
is given by 
\begin{equation*}
	\int_{-P/2}^{P/2} K_P(x-y) f(y) \, \rd y,
\end{equation*}
and more generally by 
\begin{equation*}
\int_{\mathbb{R}} K(x-y) f(y) \, \rd y
\end{equation*}
for $f$ bounded and continuous. Since the Whitham symbol $m$ is a classical symbol of order~\(-1/2\), that is,
\begin{equation*}
	\vert D_{\xi}^n m(\xi) \vert \leq C_n (1+ \vert \xi \vert )^{-1/2-n}, \qquad n \in \mathbb{N}, 
\end{equation*}
for some positive constant $C_n$ depending on $n$, the linear operator
\begin{equation}
	\label{bL}
	L\colon B_{p,q}^s (\mathbb{R}) \rightarrow B_{p,q}^{s+1/2} (\mathbb{R})
\end{equation}
is bounded \cite{BCD11}, where $B_{p,q}^s (\mathbb{R})$ denote the Besov spaces with $s \in \R$ $p,q \in [1,\infty]$; and therefore in particular on the Zygmund spaces \(\mathcal{C}^s = B^s_{\infty,\infty}\), coinciding with the H\"older spaces \(C^s\) for \(s \not\in \N\). It is bounded in a similar way on the Sobolev spaces, but we shall not use that here.

\section{Properties of solutions}\label{sec:properties}
\noindent In this section we give some a priori properties of solutions to \eqref{steadyW}. These are especially needed when taking \(P \to \infty\) in the next section. Note that the assumptions in Proposition~\ref{prop2} differ from those in for example \cite{EW19, BEP16}; and they do not require a solitary wave. Corollary~\ref{cor5} is needed in the next section to exclude degeneration of the waves, and is based on \cite{BFRW97}. It appeared also in \cite{TWW20}.

\begin{proposition}
	\label{prop1}
If $\varphi \in L^{\infty}(\mathbb{R})$ is a solution of \eqref{steadyW} with wave speed $\mu \geq 0$ and 
\[
\lim_{x \rightarrow \infty} \varphi(x) = \Phi, 
\]
then $\Phi$ solves \eqref{steadyW} with the same wave speed. In particular, $\Phi$ is either $\mu-1$ or $0$. 
\end{proposition}

\begin{proof}
By assumption and since $\Vert K \Vert_{L^1(\mathbb{R})} =1$ with \(K \geq 0\), we have 
\begin{equation*}
	\lim_{x\rightarrow \infty} L\varphi(x) = 	\lim_{x\rightarrow \infty} \int_{\mathbb{R}} K(y) \varphi(x-y) \, \rd y= \Phi,  
\end{equation*}
so that $\Phi$ is a constant solution to \eqref{steadyW} with wave speed $\mu$, that is, $\Phi$ solves $(\mu-1) \Phi= \Phi^2$. Therefore either $\Phi= \mu -1$ or $\Phi=0$. 
\end{proof}

\begin{proposition}\label{prop2}
Let $\varphi \in L^{\infty}(\mathbb{R})$ be a solution of \eqref{steadyW} with wave speed $\mu \geq 0$. If $\varphi$ is even, nonnegative, nonconstant, and nondecreasing on $(-\infty, 0)$, then $\mu \geq 1$. 
\end{proposition}

\begin{proof}
Assume that $\mu<1$. Due to assumption  $\lim_{x\rightarrow- \infty} \varphi(x)$ exists, 
and according to Proposition \ref{prop1}, the limit takes the value $0$. Because $\varphi$ solves $(\mu-\varphi)\varphi=L\varphi$, we have for $x<0$,
\begin{equation*}
(\mu-\varphi(x))\varphi(x)=  \int_{|y| < |x|} K(x-y) \varphi(y) \, \rd y +  \int_{|y| \geq |x|} K(x-y) \varphi(y) \, \rd y.  
\end{equation*}
Using that $\varphi$ is nondecreasing on $(-\infty,0)$, even, and nonnegative, as well as that $K$ 
is positive on $\mathbb{R} \backslash \{ 0\}$, we further get for $x<0$, 
 \begin{equation*}
 	(\mu-\varphi(x))\varphi(x) \geq \varphi(x) \int_0^{-2x} K(y) \, \rd y + \int_{x+s<y<x} K(x-y) \varphi(y) 
 	\, \rd y, 
 \end{equation*}
where $s<0$ is chosen in such a way that (since $\Vert K\Vert_{L^1(\mathbb{R})}=1$) 
\begin{equation}
	\label{prop2_1}
	\frac{1}{2} > \int_s^0 K(z) \, \rd z > \mu- \frac{1}{2} + \varepsilon, 
\end{equation}
where $\varepsilon>0$. Hence by defining 
\begin{equation*}
\Psi(x):= \int_0^{-2x} K(y) \, \rd y
\end{equation*}
and by the monotonicity of $\varphi$ we obtain for $x<0$,
\begin{equation}
\label{prop2_2}
	(\mu-\varphi(x))\varphi(x) \geq \varphi(x) \Psi(x) + \varphi(x+s) \int_s^0 K(z) \, \rd z.
\end{equation}
To obtain a contradiction to \(\mu < 1\), we pick a constant $\delta \in (0,1)$ so that 
\begin{equation}
	\label{prop2_3}
\frac{\mu-1/2}{\mu-1/2 + \varepsilon} < \delta <1,
\end{equation}
and consider two cases, depending on the uniformity of decay of \(\varphi\).\\[-8pt]

\noindent\textbf{Case 1.} There exists a sequence of negative real numbers $\{x_n\}_{n\in \mathbb{N}}$ with 
$\lim_{n\rightarrow \infty} x_n= - \infty$ such that $\varphi(x_n+s) \geq \delta \varphi(x_n)$ for each $n$.  In this case, \eqref{prop2_1}, \eqref{prop2_2}, and the positivity of $K$ entail that 
\begin{align*}
(\mu-\varphi(x_n))\varphi(x_n) &\geq \varphi(x_n) \Psi(x_n) + \delta \varphi(x_n) \int_s^0 K(z) \, \rd z
\\
& \geq \varphi(x_n) \Big(  \Psi(x_n) + \delta \big(\mu- \tfrac{1}{2} + \varepsilon\big)\Big),  
\end{align*}
and therefore 
\begin{equation*}
	\mu- \varphi(x_n) \geq  \Psi(x_n) + \delta \big(\mu- \tfrac{1}{2} + \varepsilon\big). 
\end{equation*}
From $\varphi(x_n) \rightarrow 0$ and $\Psi(x_n) \rightarrow 1/2$ as $n\rightarrow \infty$, it follows that 
\begin{equation*}
	\mu \geq \tfrac{1}{2} + \delta \big(\mu- \tfrac{1}{2} + \varepsilon\big), 
\end{equation*}
contradicting \eqref{prop2_3}. \\[-8pt]

\noindent \textbf{Case 2.} There exists $x_{\ast}<0$ such that for all $x \leq x_{\ast}$ the inequality $\varphi(x+s) <\delta \varphi(x)$ holds. In this case, by the monotonicity of $\varphi$, we see that 
\begin{align*}
	\int_{-\infty}^{x_{\ast}} \varphi(x) \, \rd x 
	&= \sum_{n\in \mathbb{N}} \int_{x_{\ast}+(n+1)s}^{x_{\ast}+ns} \varphi(x) \, \rd x 
	\\
	&\leq \sum_{n\in \mathbb{N}} \vert s \vert \varphi(x_{\ast}+ns)
	\\
	& \leq \vert s \vert \varphi(x_{\ast}) \sum_{n\in \mathbb{N}} \delta^n 
	= \vert s \vert \varphi(x_{\ast}) \frac{1}{1-\delta}.
 \end{align*}
Combined with the nonnegativity of $\varphi$, this shows that $\varphi \in L^1(\mathbb{R})$. Therefore, integration of the equation \eqref{steadyW} (see Proposition 4.6 in \cite{EW19})
implies that 
\begin{equation}\label{eq:the L2 equation}
	(\mu -1) \int_{\mathbb{R}} \varphi(x) \, \rd x = \Vert \varphi \Vert_{L^2(\mathbb{R})}^2
\end{equation}
and since $\varphi \not \equiv 0$ this again is a contradiction to the assumption that \(\mu < 1\). In both cases, we conclude that \(\mu \geq 1\).
\end{proof}




\noindent For Corollay~\ref{cor5} below we need the following regularisation lemma, for the proof of which 
we refer to \cite[pp. 113-114]{BFRW97}.

\begin{lemma}{\rm\cite{BFRW97}} 
\label{lemm4}
Let $f$ be a continuously differentiable function such that $\lim_{x\rightarrow \pm \infty} f(x)$ exist, and let $J$ be a nonnegative even function with unit integral satisfying $\int_{\mathbb{R}} J(y) \vert y \vert \, \rd y < \infty$. Then 
\begin{equation*}
	\lim_{R \rightarrow \infty} \int_{-R}^R (J \ast f -f ) \, \rd x =0. 
\end{equation*}
\end{lemma}
\medskip


\noindent The following result, which is the last in this section, overcomes a problem when one considers solutions that are not necessarily in \(L^1(\R)\), cf. \eqref{eq:the L2 equation}. 
A similar version was presented recently in~\cite{TWW20}, but as the proof is of interest to the limiting procedure we give it here as well. When $\varphi$ is a continuously differentiable solution to  \eqref{steadyW} with wave speed $\mu \geq 0$ and finite $\lim_{x\rightarrow \pm \infty} \varphi(x)$, one can show from Lemma~\ref{lemm4} that
\begin{equation*}
	\lim_{R \rightarrow \infty} \int_{-R}^R \varphi \big(  (\mu-1) - \varphi \big) \, \rd x =0, 
\end{equation*}
but we shall use the following somewhat stronger result.

\begin{corollary}
\label{cor5}
Let $\varphi \in L^{\infty}(\mathbb{R})$ be a solution of \eqref{steadyW} with wave speed $\mu \geq 0$ such that $\varphi$ is nonnegative and $\lim_{x\rightarrow \pm \infty} \varphi(x)$ exist. Then 
\begin{equation*}
\int_{\R}  \psi \ast \left( \varphi ((\mu-1)- \varphi) \right)  \, \rd x=0	
\end{equation*}
for every nonzero smooth and compactly supported test function $\psi$ with $\psi \geq 0$. 
In particular, only vanishing solutions have unit speed. 
\end{corollary}

\begin{proof}
Let $0 \leq \varphi \in L^{\infty}(\mathbb{R})$ solve \eqref{steadyW}. Then convolution with a nonzero smooth and compactly supported test function $\psi \geq 0$ yields 
\begin{equation}
\label{cor5_1}
	\psi\ast (K\ast \varphi- \varphi)= \psi \ast \left(  \varphi ((\mu-1)- \varphi) \right). 
\end{equation}
By the properties of convolution we obtain that 
\begin{equation}
\label{cor5_2}
\psi\ast (K\ast \varphi- \varphi) = K \ast (\psi \ast \varphi) - \psi \ast \varphi. 
\end{equation}
The term $\psi \ast \varphi$ is bounded, continuously differentiable, and by Lebesgue's dominated convergence theorem the limits $\lim_{x \rightarrow \pm \infty}\psi \ast \varphi(x)$ exist. Since the Whitham kernel $K$ is positive and even with $\int_{\mathbb{R}} K(y) \vert y \vert \, \rd y$ finite, due to $\Vert K \Vert_{L^1(\mathbb{R})}=1$ and Lemma \ref{propertiesWK}~(i), we can apply Lemma~\ref{lemm4}, which allows us to conclude from  \eqref{cor5_1} and  \eqref{cor5_2} that 
\begin{equation*}
 \lim_{R \rightarrow \infty } \int_{-R}^R  \psi \ast \left( \varphi ((\mu-1)- \varphi) \right) \rd x =  \lim_{R \rightarrow \infty } \int_{-R}^R \big( K \ast (\psi \ast \varphi) - \psi \ast \varphi \big) \, \rd x = 0. 
\end{equation*}
In the special case when $\mu=1$, one has
\begin{equation*}
\lim_{R \rightarrow \infty } \int_{-R}^R \psi \ast \varphi^2 \, \rd x = 0,
\end{equation*}
which since $\psi$ is nonnegative and nonzero enforces $\varphi \equiv 0$. As \(\psi\) is compactly supported, the limit reduces to a fixed integral.
\end{proof}

\section{The construction of a family of solitary waves}\label{sec:construction}

\noindent We have now come to the point where we shall construct our solutions. To parameterise the family, let \(\lambda\) be the \emph{relative height} of a solution, given by
\[
\varphi_\lambda(0) = \frac{\lambda \mu}{2}, \qquad \lambda \in (0,1].
\]
Note that solutions \(\varphi_\lambda\) depend on \(\mu\) as well, as we are always considering solution pairs \((\varphi,\mu)\). With the help of the relative height we can describe a family of crests continuously placed between the zero solution and a highest wave at \(\varphi(0) = \frac{\mu}{2}\). While, for a given period, the periodic theory does not guarantee uniqueness of solutions with respect to \(\lambda\), it does guarantee existence. The following is a summary of several results in \cite{EW19} (see Sections 4--6).

\begin{theorem} {\rm{\cite{EW19}}} \label{th6}
For each finite $P>0$ and each \(\lambda \in (0,1]\) there exists a \(P\)-periodic solution $\varphi_{P,\lambda} \in C^{1/2}(\mathbb{R})$ with wave speed \(\mu_{P,\lambda}\) and relative height \(\lambda\) of the steady Whitham equation \eqref{steadyW}. The solutions all have  subcritical wave speed   $0 < \mu_{P, \lambda} \leq 1$, obey the uniform bounds  
\begin{equation}
\label{eq:boundsP}
\mu_{P, \lambda} -1 \leq \varphi_{P, \lambda} \leq \varphi_{P, \lambda}(0)= \frac{\lambda \mu_{P, \lambda}}{2}, 
\end{equation}
and are even, strictly increasing  on $(-P/2,0)$ and smooth on $\R \setminus P \Z$. If $\lambda \in (0,1)$, then $\varphi_{P, \lambda} \in C^\infty(\R)$.
\end{theorem}

\noindent The following proposition is a uniform refinement of a result in \cite{EW19}, which will be used to get bounds for the limiting solution as \(P \to \infty\). It follows the structure of Lemma 5.2 in \cite{EW19}, but uses the knowledge of \(K_P\) to deduce uniformity in \(P\).

\begin{proposition}\label{prop7} 
Let $(\varphi_{P,\lambda} ,\mu_{P,\lambda})$ be as in Theorem \ref{th6}. There exists a positive constant
  \(\delta > 0\) such that 
\begin{equation*}
\frac{\mu_{P,\lambda}}{2} - \varphi_{P,\lambda}(x)  \geq \delta \vert x \vert^{1/2}, 
\end{equation*}
holds uniformly for all $P \geq 1$, $\lambda \in (0,1]$ and $|x| \leq \delta$.
\end{proposition}

\begin{proof}
Since $\varphi_{P,\lambda}$ is even it suffices to consider the case $x<0$.  
Using the evenness and periodicity of $K_P$ and $\varphi_{P,\lambda}$ we obtain by the definition of $L$ that 
\begin{align*}
&(L\varphi_{P,\lambda})(x+h)-(L\varphi_{P,\lambda})(x-h)\\ &\quad= \int_{-P/2}^0 (K_P(y-x)-K_P(y+x)) (\varphi_{P,\lambda}(y+h) - \varphi_{P,\lambda}(y-h)) \, \rd y. 
\end{align*}
The integrand is nonnegative since $K_P(y-x)-K_P(y+x) >0$ for $x,y \in (-P/2,0)$ by Lemma \ref{propertiesPWK} 
 and $\varphi_{P,\lambda}(y+h) -\varphi_{P,\lambda}(y-h)  \geq 0$ for $y \in (-P/2,0)$ and $h \in (0,P/2)$ 
 by Theorem \ref{th6}. 
For $x \in (-P/2,0)$, an application of Fatou's lemma to 
\begin{equation}
\label{eq:FatouP}
	(\tfrac{\mu_{P,\lambda}}{2}-\varphi_{P,\lambda}(x)) \varphi_{P,\lambda}'(x) = \lim_{h\rightarrow 0} \frac{(L\varphi_{P,\lambda})(x+h)-(L\varphi_{P,\lambda})(x-h)}{4h}
\end{equation}
combined with the previous identity gives 
\begin{equation}
	\label{prop7_1}
(\tfrac{\mu_{P,\lambda}}{2}-\varphi_{P,\lambda}(x)) \varphi_{P,\lambda}'(x) 	\geq \frac{1}{2} \int_{-P/2}^0 (K_P(y-x)-K_P(y+x)) \varphi_{P,\lambda}'(y) \, \rd y .
\end{equation}
For \(\lambda < 1\) and $\mu_{P,\lambda}>0$, \eqref{prop7_1} is an equality due to \eqref{steadyW} and $\varphi_{P, \lambda} \in C^\infty(\R)$. Now fix $x_1$, $x_2$ with $-P/4 < x_2<x_1<0$, let $x \in (x_2,x_1)$, and consider $\xi \in (-P/2,x_2]$. Then we infer from the monotonicity of $\varphi_{P,\lambda}$ on $(-P/2,0)$ and \eqref{prop7_1} that 
\begin{align*}
	(\tfrac{\mu_{P,\lambda}}{2}- \varphi_{P,\lambda}(\xi)) \varphi_{P,\lambda}'(x) &\geq 	(\tfrac{\mu_{P,\lambda}}{2}- \varphi_{P,\lambda}(x)) \varphi_{P,\lambda}'(x) 
	\\
	&\geq \frac{1}{2} \int_{-P/2}^0 (K_P(y-x)-K_P(y+x)) \varphi_{P,\lambda}'(y) \, \rd y 
	\\
	&\geq \frac{1}{2} \int_{x_2}^{x_1} (K_P(y-x)-K_P(y+x)) \varphi_{P,\lambda}'(y) \, \rd y 
	\\
	&= \frac{1}{2} \int_{x_2}^{x_1} (-2x) K_P'(y+\zeta)  \varphi_{P,\lambda}'(y) \, \rd y 
	\\
	& \geq -x_1 K_P'(2x_2) (\varphi_{P,\lambda}(x_1)-\varphi_{P,\lambda}(x_2)), 
\end{align*}
where $\vert \zeta \vert < \vert x \vert$ follows from the mean value theorem and where we have also used the convexity of $K_P$ on $(-P/2,0)$. Integration over $(x_2,x_1)$ in $x$ and division by $\varphi_{P,\lambda}(x_1)-\varphi_{P,\lambda}(x_2)>0$ gives 
\begin{equation*}
\frac{\mu_{P,\lambda}}{2}- \varphi_{P,\lambda}(\xi) \geq -x_1 K_P'(2x_2) (x_1-x_2). 
\end{equation*}
By taking $\xi=x_2=x$ and $x_1=x/2$ with $x \in (-P/4,0)$ we get that 
\begin{equation}
\label{prop7_2_2}
\frac{\mu_{P,\lambda}}{2}- \varphi_{P,\lambda}(x) \geq \frac{1}{4} \, x^2 \, K_P'(2x).
\end{equation}
Next, in view of the definition of $K_P$ and Lemma \ref{propertiesWK}~(ii), we have 
\begin{align*}
	K_P &=K+ \sum_{n\in \mathbb{Z}\backslash \{0\}} K(\cdot + nP) 
	\\
	&= \frac{1}{\sqrt{2\pi \vert \cdot \vert}} + K_{\text{reg}} + \sum_{n\in \mathbb{Z}\backslash \{0\}} K(\cdot + nP). 
\end{align*}
Choose  $0<-x \leq 1/2$ so that \(P \geq 1 \geq 2|x|\) always holds. We then infer again from Lemma~\ref{propertiesWK} (ii) that
\begin{align*}
K_P'(x) &= \frac{1}{ \sqrt{8 \pi} \, \vert x\vert^{3/2}} + K_{\text{reg}}'(x) + \sum_{n\in \mathbb{Z}\backslash \{0\}} K'(x+ nP) 
\\
&\geq \frac{1}{ \sqrt{8 \pi} \, \vert x\vert^{3/2}} -c +  \sum_{n\in \mathbb{Z}\backslash \{0\}} K'(x+ nP) 
\end{align*}
with $c>0$ being a universal constant independent of $P$ and $\lambda$, modified whenever needed. With Lemma \ref{propertiesWK} (i) we further obtain for a fixed $s_0 \in (0,\pi/2)$ that
\begin{align*}
K_P'(x) &\geq   \frac{1}{ \sqrt{8 \pi} \, \vert x\vert^{3/2}} - c \Bigg( 1 + \sum_{n\in \mathbb{Z}\backslash \{0\}} \exp(-s_0 \vert x+nP\vert)\Bigg)
\\
& \geq  \frac{1}{ \sqrt{8 \pi} \, \vert x\vert^{3/2}} - c \Bigg( 1+  \sum_{n\in \mathbb{Z}\backslash \{0\}} \exp\big(\tfrac{-s_0 P \vert n \vert}{2} \big)\Bigg) 
\\
& \geq \frac{1}{ \sqrt{8 \pi} \, \vert x\vert^{3/2}} - c,
\end{align*}
again for $0<-x \leq 1/2$ and $P\geq 1$.
By combining this with \eqref{prop7_2_2} and picking \(\delta \in (0,c)\) small enough, one obtains the desired bound 
\begin{equation*}
	\frac{\mu_{P,\lambda}}{2}- \varphi_{P,\lambda}(x) \geq \delta \vert x \vert^{1/2},
\end{equation*}
uniformly for $P\geq 1$,  $\lambda \in (0,1]$ and $|x| \in [0,\delta]$ by continuity.
\end{proof}

\noindent The following theorem contains the main convergence result, but as we shall see, does not yield solitary waves directly, as we do not the control the asymptotic behaviour of these waves. In fact, since all periodic solutions are sign-changing, one could imagine different situations appearing, and only Corollary~\ref{cor:galilean} below resolves this.

\begin{theorem}\label{thm:main}
Fix $\lambda \in (0,1]$.
Given a sequence $\{ (\varphi_{P_n,\lambda}, \mu_{P_n,\lambda})\}_{n}$ of solution pairs from Theorem \ref{th6} with \(P_n \to \infty\) as \(n \to \infty\), there exists a subsequence converging locally uniformly as $P_n \to \infty$ to a solution pair 
\[
(\phi_\lambda, \nu_\lambda)  \in C(\mathbb{R}) \times [0,1]
\]
with relative height $\lambda$ and 
satisfying 
\[
\nu_\lambda -1 \leq \phi_\lambda \leq \frac{\lambda \nu_\lambda}{2}. 
\]
The locally uniform limit $\phi_{\lambda}$ is even, nonconstant, nonperiodic, 
strictly increasing with \(\phi_\lambda' > 0\) on $(-\infty,0)$, and smooth on $\mathbb{R}$ in the case of \(\lambda \in (0,1)\) and $\nu_\lambda >0$.
\end{theorem}

\begin{remark}
We shall see later that in fact \(\nu_\lambda > 0\) for all \(\lambda \in (0,1]\). One could build an argument as in Corollary~6.11 in \cite{EW19}, by establishing uniform lower bounds on the wave speed for all bifurcation curves for large \(P\) (they are not uniform for small \(P\), and even for large P the local slope of the curves is not uniform), but we choose to handle this after the limiting procedure instead, with the use of results from Section~\ref{sec:properties}.
\end{remark}

\begin{proof}
Let $\lambda \in (0,1]$ be fixed.  By assumption, there exists a (not relabeled) subsequence \(\{P_n\}_n\) such that $\lim_{P_n \rightarrow \infty} \mu_{P_n,\lambda}$ converges to a value \(\nu_{\lambda} \in [0,1]\).  For each general
$P \in (0,\infty)$, we have by \eqref{steadyW}
\begin{equation}
\label{th8_1}
(\mu_{P,\lambda} - \varphi_{P,\lambda}(x)- \varphi_{P,\lambda}(y)) (\varphi_{P,\lambda}(x)- \varphi_{P,\lambda}(y))= L\varphi_{P,\lambda}(x) - L\varphi_{P,\lambda}(y)
\end{equation}
for all $x, y \in \mathbb{R}$. 
Since  $\varphi_{P,\lambda} \leq \frac{\lambda \mu_{P,\lambda}}{2} \leq \frac{\mu_{P,\lambda}}{2}$, it follows that
\begin{equation*}
\mu_{P,\lambda} - \varphi_{P,\lambda}(x)- \varphi_{P,\lambda}(y) \geq \varphi_{P,\lambda}(x)- \varphi_{P,\lambda}(y),
\end{equation*}
and hence from \eqref{th8_1} that 
\begin{equation*}
(\varphi_{P,\lambda}(x) -\varphi_{P,\lambda}(y))^2 \leq \vert L \varphi_{P,\lambda}(x) - L\varphi_{P,\lambda}(y) \vert, 
\end{equation*}
again for all $x, y \in \mathbb{R}$. Then, since $L$ is a bounded map from $L^{\infty}(\mathbb{R}) \hookrightarrow B^0_{\infty,\infty}(\mathbb{R})$ to $B^{1/2}_{\infty,\infty} (\mathbb{R}) = C^{1/2}(\mathbb{R})$ according to \eqref{bL}, we obtain that 
\begin{align*}
	(\varphi_{P_n,\lambda}(x) -\varphi_{P_n,\lambda}(y))^2 &\leq \Vert L\Vert_{\mathcal{L}(L^{\infty}(\mathbb{R}) , C^{1/2}(\mathbb{R}))} \, \Vert \varphi_{P_n,\lambda} \Vert_{\infty} \, 	\vert x-y \vert^{1/2}\\
	&\leq  \Vert L\Vert_{\mathcal{L}(L^{\infty}(\mathbb{R}) , C^{1/2}(\mathbb{R}))} 	\vert x-y \vert^{1/2},
\end{align*}
in view of  \eqref{eq:boundsP} in Theorem~\ref{th6}. As this holds for all $x,y \in \mathbb{R}$, it guarantees that $\{ \varphi_{P_n,\lambda}\}_{n}$ is uniformly bounded in $C^{1/4}(\mathbb{R})$. In particular, $\{ \varphi_{P_n,\lambda}\}_{n}$ is an equicontinuous family of solutions, and thus, by the Arzel\`{a}--Ascoli theorem and a diagonal argument, has a subsequence converging locally uniformly to a function $\phi_{\lambda} \in C(\mathbb{R})$. Also, $\phi_{\lambda}$ is bounded below by $\nu_{\lambda} -1$ and above by $\lambda \nu_{\lambda}/2$ due to  \eqref{eq:boundsP}. In addition, 
\[
\phi_{\lambda}(0) = \frac{\lambda\nu_{\lambda}}{2}.
\]
Because $|\varphi_{P_n,\lambda}| \leq 1$ and $\Vert K \Vert_{L^1(\mathbb{R}) }=1$, Lebesgue's dominated convergence theorem yields 
\begin{equation*}
	L\varphi_{P_n,\lambda}(x) = K \ast \varphi_{P_n,\lambda}(x) \rightarrow K \ast \phi_{\lambda}(x) = L\phi_{\lambda}(x)
\end{equation*}
as $P_n \rightarrow \infty$ for all $x \in \mathbb{R}$, and it thus follows that $\phi_{\lambda}$ solves \eqref{steadyW} with wave speed $\nu_{\lambda}$. 

Since $\varphi_{P_n,\lambda}$ is even and strictly increasing on $(-P_n/2,0)$, its locally uniform limit $\phi_{\lambda }$ inherits evenness and is at least nondecreasing on $(-\infty,0)$. In order to establish that $\phi_{\lambda }$ is strictly increasing on $(-\infty,0)$, we first prove that $\phi_{\lambda }$ is nonconstant. For this we note that in view of Proposition \ref{prop7}, we obtain in the limit $P_n \rightarrow \infty$ the inequality 
\begin{equation}
\label{th8_2}
\phi_{\lambda } (x) \leq \frac{\nu_{\lambda }}{2} - \delta \vert x \vert^{1/2}
\end{equation}
near the origin. In the case when \(\lambda = 1\) and \(\phi_{\lambda} (0) = \frac{\nu_{\lambda }}{2}\) this excludes constant solutions. When \(\lambda \in (0,1)\), we instead use that $\varphi = 0$ and $\varphi = \nu_{\lambda}-1$ are the only constant solutions to \eqref{steadyW}. The first case would force \(\nu_\lambda = 0\) which again would be a contradiction to \eqref{th8_2}. The second case would instead yield 
\[
\lambda \frac{\nu_\lambda}{2} = \nu_\lambda - 1,
\]
which is impossible for \(\lambda \in (0,1]\), \(\nu_\lambda \in [0,1]\). Then clearly $\phi_{\lambda}$ is nonconstant, which combined with monotonicity gives that it is also nonperiodic.
 
To show that $\phi_{\lambda}$ is strictly increasing on $(-\infty,0)$ is direct. Due to evenness of both $K$ and $\phi_{\lambda}$,  one has
\begin{align}
\label{th8_3}
(L\phi_{\lambda})(x+h) &- (L\phi_{\lambda})(x-h) \nonumber \\ &= \int_{-\infty}^0 (K(y-x)-K(y+x))(\phi_{\lambda}(y+h) - \phi_{\lambda}(y-h)) \, \rd y. 
\end{align}
For $x<0$ and $h>0$, the first factor of the integrand is 
positive while the second one is nonnegative, and since $\phi_{\lambda}$ is nonconstant, we conclude that 
$(L\phi_{\lambda})(x+h) > (L\phi_{\lambda})(x-h)$ whenever $x$, $h$ are chosen as above. From \eqref{steadyW} we have 
\begin{equation}
\label{th8_4}
L\phi_{\lambda}(x) - L\phi_{\lambda}(y) = (\nu_{\lambda}- \phi_{\lambda}(x) - \phi_{\lambda}(y)) (\phi_{\lambda}(x)- \phi_{\lambda}(y))
\end{equation}
and using $\phi_{\lambda}\leq \frac{\lambda\nu_{\lambda}}{2}\leq \frac{\nu_{\lambda}}{2}$, we see that the first expression on the right-hand side is nonnegative with equality only when $\phi_{\lambda}(x) = \phi_{\lambda}(y) = \frac{\nu_{\lambda}}{2}$. Since 
we already verified that $L \phi_{\lambda}$ is strictly increasing on $(-\infty, 0)$, the fact that $\phi_{\lambda}$ 
is nondecreasing together with \eqref{th8_4} imply that $\phi_{\lambda}$ is indeed strictly increasing on 
$(-\infty,0)$ with strict maximum at \(\phi_\lambda(0) = \lambda \nu_\lambda/2\).

A proof for that  \(\phi_{\lambda}\) is smooth on \(\R\) when \(\lambda \in (0,1)\) and $\nu_\lambda >0$, can be found in \cite[Theorem 5.1]{EW19}. To see that $\phi_{\lambda}$ has an everywhere positive derivative on $(-\infty,0)$, differentiate or apply Fatou's lemma (for the case \(\lambda = 1\)) to 
\begin{equation*}
(\nu_{\lambda}-2\phi_{\lambda}(x)) \phi_{\lambda}'(x) = \lim_{h\rightarrow 0} \frac{ (L\phi_{\lambda})(x+h) - 
(L\phi_{\lambda})(x-h)}{2h}
\end{equation*}
and use \eqref{th8_3} to get 
\begin{equation*}
(\nu_{\lambda}-2\phi_{\lambda}(x)) \phi_{\lambda}'(x) \geq \int_{-\infty}^0 (K(y-x) - K(y+x)) \, \phi_{\lambda}'(y) 
\, \rd y. 
\end{equation*}
Since \(K(y-x) - K(y+x) >0\) here, and $\phi_{\lambda}$ is nonconstant, this enforces 
\begin{equation*}
(\nu_{\lambda}-2\phi_{\lambda}(x)) \phi_{\lambda}'(x)  > 0 
\end{equation*}
on the negative half-line and therefore $\phi_{\lambda}' >0$ on the same interval. 
\end{proof}

As it turns out, the obtained solutions are \emph{not} solitary waves, but decay to the non-zero line of trivial solutions in the Whitham equation, namely \(\phi \equiv \nu -1\). The following corollary establishes that and excludes the case \(\nu_\lambda = 1\), to obtain the final waves through a transformation.

\begin{corollary}\label{cor:galilean} 
Given $\lambda \in (0,1]$, let  $(\phi_{\lambda}, \nu_{\lambda})$ be a limiting solution as in Theorem~\ref{thm:main}. Then the Galilean transformation
\begin{equation}\label{eq:galilean}
\varphi_\lambda = \phi_{\lambda} + 1 - \nu_{\lambda}, \qquad  \mu_{\lambda} = 2- \nu_{\lambda}
\end{equation}
defines a positive, bounded, even and nonconstant solution of \eqref{steadyW}, which is strictly decreasing on $(0,\infty)$. The resulting solitary wave has supercritical wave speed $\mu_{\lambda} \in (1,2)$, and satisfies 
\[
0 <  \varphi_{\lambda} \leq  \mu_\lambda -1 + \lambda \left(1-\frac{\mu_{\lambda}}{2} \right), 
\]
with equality at \(x = 0\).  Moreover, 
\[
\lim_{x\rightarrow \pm \infty} \varphi_{\lambda}(x)= 0,
\]
\[
e^{\eta \vert \cdot \vert} \, 
 \varphi_{\lambda}  \in L^1(\R) \cap L^{\infty}(\R),
\] 
for some \(\eta = \eta(\lambda) > 0\). The solution is everywhere smooth for \(\lambda \in (0,1)\) while for 
\(\lambda =1\)  the wave of maximal wave height $\varphi_1$ is  \(C^{1/2}\) and smooth on 
 \(\R\setminus\{0\}\). 
\end{corollary}

\begin{proof}
Since, according to Theorem \ref{thm:main}, $\phi_{\lambda}$ is even, monotone on $(-\infty,0)$ and bounded, the limits $\lim_{x\rightarrow \pm \infty} \phi_{\lambda}(x)$ exist and, due to Proposition \ref{prop1}, are  equal to $0$ or $\nu_{\lambda}-1$. We rule out the first case.

If $\lim_{x\rightarrow \pm \infty} \phi_{\lambda}(x) = 0$, then by strict monotonicity, \(0 < \phi_\lambda \leq \lambda \nu_\lambda/2\). By the bounds for the wave speed given in Proposition~\ref{prop2} and Theorem~\ref{thm:main}, we must have \(\nu_\lambda = 1\). But, as \(\phi_\lambda\) is bounded, Corollary~\ref{cor5} forces \(\phi_\lambda = 0\), which is a contradiction.

Therefore,
\[
\lim_{x\rightarrow \pm \infty} \phi_{\lambda}(x) = \nu_\lambda -1 < 0,
\]
and we can apply the Galilean transformation \eqref{eq:galilean} to obtain a positive solution with wave speed \(\mu_\lambda \in [1,2]\). Again, by Corollary~\ref{cor5}, the case \(\mu_\lambda =1\) is excluded, and we have a solution with $\mu_{\lambda} > 1$  and $\lim_{x\rightarrow \pm \infty} \varphi_{\lambda}(x)= 0$.  Hence we can invoke the a priori results from \cite[Proposition 3.13]{BEP16}, to conclude exponential decay  
\[
e^{\eta \vert \cdot \vert} \, 
 \varphi_{\lambda}  \in L^1(\R) \cap L^{\infty}(\R)
 \]
 of the solutions. Here, \(\eta > 0\) is a decay factor depending on the wave speed in a non-trivial way, see \cite{BEP16}. The case $\mu_\lambda =2$ is excluded as well, as Corollary~\ref{cor5} then would enforce
	\[
	\int_{\R} \psi \ast (\varphi_\lambda (1-\varphi_\lambda)) \, \rd x=0, 
	\]
for an abritrary positive test function \(\psi\). But $\varphi_\lambda \equiv 0$ and $\varphi_\lambda \equiv 1$ are both contradictions to $\varphi_\lambda$ being nonconstant, so  $\mu_\lambda <2$. The smoothness on \(\R\setminus\{0\}\)  follows from \cite[Theorem 5.1]{EW19} and by \cite[Theorem 5.4]{EW19} we get \(C^{1/2}\)-regularity exactly for  $\lambda =1$. When \(\lambda \in (0,1)\),  the Galilean transformation yields $\nu_{\lambda}>0$, and by Theorem~\ref{thm:main} we then obtain that $\phi_{\lambda}$ and $\varphi_{\lambda}$ are smooth on $\R$. 
\end{proof}

We present finally three results that give additional information about the solutions, the first of which is similar to the periodic result in \cite{BP21}, although the proof is different. Proposition~\ref{prop:touching} is for general solutions, and adds the point (ii). Unfortunately, we have not been able to rule out the case of two general solutions of the same wave height or speed, but see Proposition~\ref{prop:injective} for injectivity of the solutions constructed in this paper.

\begin{proposition}[improved touching lemma]\label{prop:touching}
Let \((\varphi, \mu)\) and \((\tilde\varphi, \tilde \mu)\) be two solution pairs of the steady Whitham equation \eqref{steadyW} with \(\mu, \tilde \mu \in (1,2)\) and \(\varphi, \tilde \varphi\) positive solutions with images in \((0,\frac{\mu}{2}]\). By symmetry, assume \(\mu \geq \tilde \mu\). 
\begin{itemize}
\item[(i)] If \(\tilde \varphi(x) \geq \varphi(x)\) for all \(x \in \R\) with equality at some point, then \(\tilde\varphi \equiv \varphi\) are identical.
\item[(ii)] If \(\tilde\varphi - \varphi\) is sign-shifting, its minimum is obtained where \(\mu - \tilde\varphi - \varphi < 1\).
\end{itemize}
\end{proposition}

\noindent Note that for the solutions constructed in this paper, Corollary~\ref{cor:galilean} implies that the quantity \(c(x) = \mu - \varphi(x) - \tilde\varphi(x)\) used in the proof of Proposition~\ref{prop:touching}~(ii) will satsify
\[
c(0) = \mu - 2\left( \mu - 1 + \lambda\left(1-\frac{\mu}{2}\right)\right) = (1-\lambda)(2-\mu),
\]
when \(\mu = \tilde \mu\). Because \(\lambda \in (0,1)\) and \(\mu = \mu_\lambda \in (1,2)\), the case of \((ii)\) is not excluded. As the waves are solitary, one however has \(c(x) > 1\) for all \(|x| \geq x_\lambda\) for some \(x_\lambda\).

\begin{proof}
As \(-\mu\varphi + \varphi^2 + L \varphi = 0\) and  \(-\tilde\mu\tilde\varphi + \tilde\varphi^2 + L \tilde\varphi = 0\), we have similar to \eqref{th8_1} that 
\begin{equation}\label{eq:two phi}
(\mu - \varphi - \tilde\varphi) (\tilde \varphi - \varphi) + (\tilde\mu - \mu)\tilde\varphi = L(\tilde\varphi - \varphi).
\end{equation}
Let \(\psi = \tilde \varphi - \varphi\) and \(c(x) = \mu - \varphi(x) - \tilde\varphi(x) \geq 0\). When \(\varphi\) and \(\tilde\varphi\) are bell-shaped solitary waves, \(x \mapsto c(x)\) is monotone on a half-line with minimum at \(x = 0\). From the equation \eqref{eq:two phi} and  \(\mu \geq \tilde \mu\) we obtain
\begin{align*}
c(x) \psi(x)  \geq c(x) \psi(x) + (\tilde\mu - \mu)\tilde\varphi(x) = L \psi(x).
\end{align*}
Assume now further that  \(\psi(x) = \tilde \varphi(x) - \varphi(x) \geq 0\) for all \(x \in \R\) as in the case (i). Then
\begin{align*}
c(x)\psi(x) &\geq  \int_\R K(x-y) \psi(y) \, \rd y\\
&\geq   \int_{x_1}^{x_2} K(x-y) \psi(y) \, \rd y\\
&\geq   \min_{[x_1,x_2]} \psi \int_{x_1}^{x_2} K(x-y) \, \rd y,
\end{align*}
for any \(x, x_1, x_2 \in \R\) with $x_1\leq x_2$. If now \(\psi(x) = 0\) at some point, fix that point. As \(K\) is strictly positive, we may choose the interval \([x_1, x_2]\) arbitrarily, yielding \(\min_{[x_1,x_2]} \psi=0\) for all intervals, and thus by continuity that \(\psi \equiv 0\).

In the more general case (ii), let \(\psi_\text{min} = \min_{\R} \psi\). By using again that \(\mu \geq \tilde \mu\) and adding and subtracting \(c(x) \psi_\text{min}\), one then obtains
\begin{align*}
c(x)(\psi(x) - \psi_\text{min})  + c(x) \psi_\text{min} &\geq  \int_\R K(x-y) (\psi(y) - \psi_\text{min}) \, \rd y + \int_\R K(x-y) \psi_\text{min} \, \rd y\\
&=\int_\R K(x-y) (\psi(y) - \psi_\text{min}) \, \rd y +  \psi_\text{min},
\end{align*}
and consequently
\begin{align*}
c(x)(\psi(x) - \psi_\text{min})  + (c(x) -1)\psi_\text{min} &\geq   \int_\R K(x-y) (\psi(y) - \psi_\text{min}) \, \rd y\\
&\geq   \int_{x_1}^{x_2} K(x-y) (\psi(y) - \psi_\text{min}) \, \rd y\\
&\geq   \min_{[x_1,x_2]}(\psi - \psi_\text{min}) \int_{x_1}^{x_2} K(x-y) \, \rd y.
\end{align*}
If \(c(x) \geq 1\) is realised where \(\psi(x) = \psi_\text{min}\), the same argument as above yields \(\psi \equiv \psi_\text{min}\) and then by \(\lim_{|x| \to \infty} \psi(x) = 0\), that \(\psi \equiv 0\). So the only other possibility is that \(c(x) < 1\) at the point of the \(\min_\R \psi\) (which has to be reasonable close to \(x = 0\), see the remark after Proposition~\ref{prop:touching}).
\end{proof}

\begin{proposition}\label{prop:injective}
The branch of solitary solutions from Corollary~\ref{cor:galilean} is injective as a mapping \(\lambda \mapsto (\varphi_\lambda, \mu_\lambda) \in C(\R) \times \R\).
\end{proposition}

\begin{proof}

Let \(\alpha = \frac{\varphi_\lambda(0)}{\mu_\lambda/2}\) denote the relative height of solutions from Corollary~\ref{cor:galilean}.  The parameter \(\lambda\) determines the relative height of the originally constructed solutions \(\phi_\lambda\) with wavespeed \(\nu_\lambda \in (0,1)\). As the height of the new solutions \(\varphi_\lambda\) is
\(
\varphi_\lambda(0) = \mu_\lambda -1 + \lambda \left(1-\frac{\mu_{\lambda}}{2} \right), 
\)
we need only solve 
\[
\alpha = \frac{\mu -1 + \lambda \left(1-\frac{\mu}{2} \right)}{\mu/2}
\]
for \(\lambda\) to obtain that 
\[
\lambda = \frac{2 - (2-\alpha)\mu}{2-\mu}.
\]
This means that two identical solutions on the curve enforce the same \(\lambda\), hence the mapping \(\lambda \mapsto (\alpha_\lambda, \mu_\lambda)\) injective, and thus also \(\lambda \mapsto (\varphi_\lambda, \mu_\lambda) \in C(\R) \times \R\).
\end{proof}

\begin{proposition}\label{prop:upperbound}
With \(\alpha = \frac{\varphi_\lambda(0)}{\mu_\lambda/2}\) the relative height of solutions from Corollary~\ref{cor:galilean}, the wave speed always satisfies 
\[
\mu_\lambda < \frac{2}{2- \alpha}.
\]
For \(\alpha\) close enough to \(1\) (highest wave), the wave speed is uniformly bounded away from \(\mu = 2\) by a positive constant depending only on the regularity estimate from Proposition~\ref{prop7}.
\end{proposition}

\noindent The bounds in Proposition~\ref{prop7} hold for all continuous, non-trivial, bell-shaped solutions of \eqref{steadyW} satisfying \(\varphi \leq \frac{\alpha\mu}{2}\) (there are no other solutions known of the Whitham equation). It is possible that one can could improve the upper bound for \(\mu\) using precise estimates for the kernel \(K\) and the optimal decay rate from \cite{BEP16, MR4324298} in combination with new optimal constants found in ongoing work \cite{EMV21}, but it is probably hard bordering on the impossible to determine an exact upper bound.

\begin{proof}
The trivial estimate 
\begin{align*}
\left( \mu - \varphi(0) \right) \varphi(0) &= L\varphi(0) = \int_\R K(y) \varphi(y)\, \rd y < \frac{\alpha\mu}{2},
\end{align*}
shows that \(\mu(1-\frac{\alpha}{2}) \frac{\alpha \mu}{2} <  \frac{\alpha\mu}{2}\), which directly yields \(\mu < \frac{2}{2-\alpha}\). We now use this estimate as well as the a priori estimate from Proposition~\ref{prop7}, writing 
\begin{align*}
\left( \mu - \varphi(0) \right) \varphi(0) &= L\varphi(0) = \int_\R K(y) \varphi(y)\, \rd y\\ 
&= \int_{|y| \leq \delta} K(y) \varphi(y)\, \rd y + \int_{|y| > \delta} K(y) \varphi(y)\, \rd y\\ 
&\leq \int_{|y| \leq \delta} K(y) \left( \frac{\mu}{2} - \delta |y|^{1/2} \right)\, \rd y + \int_{|y| > \delta} K(y) \frac{\alpha\mu}{2}\, \rd y\\ 
&\leq \frac{\alpha\mu}{2} + \left(\frac{1 - \alpha}{2-\alpha} - \delta\right) \int_{|y| \leq \delta} K(y) |y|^{1/2} \, \rd y.
\end{align*}
Because \(\delta\) is uniform for all solutions, considering \(\alpha\) close enough to \(1\) yields a negative last term, so that
\[
\mu\left(1-\frac{\alpha}{2}\right) \frac{\alpha \mu}{2} -  \frac{\alpha\mu}{2} \leq - C_\delta^2.
\]
This shows that \(\mu\) is absolutely bounded away from \(\frac{2}{2- \alpha}\) with a distance depending only on \(\delta\) for all such solutions.
\end{proof}

\section*{Acknowledgements}
\noindent The authors would like to thank E. Wahlén for drawing our attention to the investigation~\cite{BFRW97}.

\bibliographystyle{siam}
\bibliography{ENW}

\begin{thebibliography}{10}

\bibitem{MR4324298}
{\sc M.~N. Arnesen}, {\em Decay and symmetry of solitary waves}, J. Math. Anal.
  Appl., 507 (2022).

\bibitem{BCD11}
{\sc H.~Bahouri, J.-Y. Chemin, and R.~Danchin}, {\em Fourier analysis and
  nonlinear partial differential equations}, vol.~343 of Grundlehren der
  Mathematischen Wissenschaften, Springer, Heidelberg, 2011.

\bibitem{BFRW97}
{\sc P.~W. Bates, P.~C. Fife, X.~Ren, and X.~Wang}, {\em Traveling waves in a
  convolution model for phase transitions}, Arch. Rational Mech. Anal., 138
  (1997), pp.~105--136.

\bibitem{MR3430140}
{\sc H.~Borluk, H.~Kalisch, and D.~P. Nicholls}, {\em A numerical study of the
  {W}hitham equation as a model for steady surface water waves}, J. Comput.
  Appl. Math., 296 (2016), pp.~293--302.

\bibitem{BEP16}
{\sc G.~Bruell, M.~Ehrnstr\"{o}m, and L.~Pei}, {\em Symmetry and decay of
  traveling wave solutions to the {W}hitham equation}, J. Differential
  Equations, 262 (2017), pp.~4232--4254.

\bibitem{BP21}
{\sc G.~Bruell and L.~Pei}, {\em Symmetry of periodic traveling waves for
  nonlocal dispersive equations}.
\newblock arXiv:2101.05739.

\bibitem{Buffoni04a}
{\sc B.~Buffoni}, {\em Existence and conditional energetic stability of
  capillary-gravity solitary water waves by minimisation}, Arch. Ration. Mech.
  Anal., 173 (2004), pp.~25--68.

\bibitem{EGW11}
{\sc M.~Ehrnstr{\"o}m, M.~D. Groves, and E.~Wahl\'en}, {\em On the existence
  and stability of solitary-wave solutions to a class of evolution equations of
  {W}hitham type}, Nonlinearity, 25 (2012), pp.~1--34.

\bibitem{EK08}
{\sc M.~Ehrnstr{\"o}m and H.~Kalisch}, {\em {Traveling waves for the Whitham
  equation}}, Differential Integral Equations, 22 (2009), pp.~1193--1210.

\bibitem{EMV21}
{\sc M.~{Ehrnstr{\"o}m}, O.~I.~H. Maehlen, and K.~Varholm}, {\em {On the
  precise behaviour of extreme solutions to uni- and bidirectional Whitham
  equations}}.
\newblock In preparation.

\bibitem{EW19}
{\sc M.~Ehrnstr\"{o}m and E.~Wahl\'{e}n}, {\em On {W}hitham's conjecture of a
  highest cusped wave for a nonlocal dispersive equation}, Ann. Inst. H.
  Poincar\'{e} Anal. Non Lin\'{e}aire, 36 (2019), pp.~1603--1637.

\bibitem{EW20whitham}
{\sc M.~Ehrnstr\"{o}m and Y.~{Wang}}, {\em {Enhanced existence time of
  solutions to evolution equations of Whitham type}}.
\newblock To appear in Discrete Contin. Dyn. Syst. arXiv:2008.12722.

\bibitem{MR4321411}
{\sc L.~Emerald}, {\em Rigorous derivation of the {W}hitham equations from the
  water waves equations in the shallow water regime}, Nonlinearity, 34 (2021),
  pp.~7470--7509.

\bibitem{MR4072387}
{\sc F.~Hildrum}, {\em Solitary waves in dispersive evolution equations of
  {W}hitham type with nonlinearities of mild regularity}, Nonlinearity, 33
  (2020), pp.~1594--1624.

\bibitem{MR3682673}
{\sc V.~M. Hur}, {\em Wave breaking in the {W}hitham equation}, Adv. Math., 317
  (2017), pp.~410--437.

\bibitem{MR3298879}
{\sc V.~M. Hur and M.~A. Johnson}, {\em Modulational instability in the
  {W}hitham equation for water waves}, Stud. Appl. Math., 134 (2015),
  pp.~120--143.

\bibitem{JTW21}
{\sc M.~A. {Johnson}, T.~Truong, and M.~H. Wheeler}, {\em {Solitary waves in a
  Whitham equation with small surface tension}}.
\newblock arXiv:2103.02675.

\bibitem{MR4057934}
{\sc M.~A. Johnson and J.~D. Wright}, {\em Generalized solitary waves in the
  gravity-capillary {W}hitham equation}, Stud. Appl. Math., 144 (2020),
  pp.~102--130.

\bibitem{MR3763731}
{\sc C.~Klein, F.~Linares, D.~Pilod, and J.-C. Saut}, {\em On {W}hitham and
  related equations}, Stud. Appl. Math., 140 (2018), pp.~133--177.

\bibitem{MR3060183}
{\sc D.~Lannes}, {\em {The Water Waves Problem}}, vol.~188 of Mathematical
  Surveys and Monographs, American Mathematical Society, Providence, RI, 2013.
\newblock Mathematical analysis and asymptotics.

\bibitem{MR3188389}
{\sc F.~Linares, D.~Pilod, and J.-C. Saut}, {\em Dispersive perturbations of
  {B}urgers and hyperbolic equations {I}: {L}ocal theory}, SIAM J. Math. Anal.,
  46 (2014), pp.~1505--1537.

\bibitem{MR4061635}
{\sc A.~Stefanov and J.~D. Wright}, {\em Small {A}mplitude {T}raveling {W}aves
  in the {F}ull-{D}ispersion {W}hitham {E}quation}, J. Dynam. Differential
  Equations, 32 (2020), pp.~85--99.

\bibitem{TWW20}
{\sc T.~Truong, E.~Wahl\'{e}n, and M.~H. Wheeler}, {\em {Global bifurcation of
  solitary waves for the Whitham equation}}.
\newblock arXiv: 2009.05713v2, 2020.

\bibitem{MR766131}
{\sc R.~E.~L. Turner}, {\em A variational approach to surface solitary waves},
  J. Differential Equations, 55 (1984), pp.~401--438.

\end{thebibliography}

\end{document}